\documentclass{amsart}
\usepackage{amssymb}   
\usepackage{amsmath}
\usepackage{amsthm}
\usepackage[all,dvips]{xy}
\usepackage[]{pstricks,pst-node,pst-plot}
\usepackage{graphicx}
\usepackage{pstcol,pst-char}
\def\a={{\buildrel a \over =}}
\def\na={{\buildrel a \over \neq}}
\def\CC{{\mathbb C}}
\def\PP{{\mathbb P}}

\def\Tcal{{\mathcal T}}

\def\ZZ{{\mathbb Z}}

\def\Lg{\Lambda}

\def\bM{\overline{M}}
\def\bH{\overline{H}}

\newtheorem{theorem}{Theorem}[section]
\newtheorem{lemma}[theorem]{Lemma}
\newtheorem{proposition}[theorem]{Proposition}
\newtheorem{corollary}[theorem]{Corollary}

\newtheorem{definition-lemma}[theorem]{Definition-Lemma}

\theoremstyle{definition}

\theoremstyle{remark}
\newtheorem{remark}[theorem]{\bf Remark}

\def\vandaag{\number\day\space\ifcase\month\or
 januari\or februari\or  maart\or  april\or mei\or juni\or  juli\or
 augustus\or  september\or  oktober\or november\or  december\or\fi,
\number\year}
\def\today{\ifcase\month\or
 Jan\or Febr\or  Mar\or  Apr\or May\or Jun\or  Jul\or
 Aug\or  Sep\or  Oct\or Nov\or  Dec\or\fi
 \space\number\day, \number\year}

\begin{document}

\title[The Class of a Hurwitz Divisor]
{The Class of a Hurwitz Divisor on the Moduli of Curves of Even Genus}
\author{Gerard van der Geer $\,$}
\address{Korteweg-de Vries Instituut, Universiteit van
Amsterdam, Postbus 94248, 1090 GE Amsterdam,  The Netherlands}
\email{G.B.M.vanderGeer@uva.nl}
\author{$\,$ Alexis Kouvidakis}
\address{Department of Mathematics, University of Crete,
GR-71409 Heraklion, Greece}
\email{kouvid@math.uoc.gr}
\begin{abstract}
We calculate the cycle class of the Hurwitz divisor $D_2$ on ${\bM}_g$
for $g=2k$ given by the degree $k+1$ covers of ${\PP}^1$ with
simple ramification points, two of which lie in the same fibre.
We also study some aspects of
the geometry of the natural map from the Hurwitz space
$\overline{H}_{2k,k+1}$ to the moduli space $\overline{M}_{2k}$.
\end{abstract}
\maketitle
\begin{section}{Introduction}\label{sec:intro}
Hurwitz cycles are playing a significant role in the study of
the geometry of the moduli space $M_g$ of curves of genus $g$. For example,
they appeared prominently in the work \cite{HMu} of Harris and Mumford
on the Kodaira dimension of $M_g$. 
Faber and Pandharipande showed in \cite{FP} 
that the cycle classes of Hurwitz loci are tautological. 
For some Hurwitz loci the cycles classes
are known, but for many such loci the cycle classes are still unknown;
for work in this direction see \cite{GF}.

We work over the complex numbers.
The generic curve of even genus $g=2k$ is in finitely many
ways a degree $k+1$ cover of the projective line with $6k$ branch
points. By the Hurwitz-Zeuthen formula these are all simple branch points.
Here simple branch point means that the corresponding fiber
has exactly one simple ramification point. 
The condition that two of the resulting $6k$ ramification 
points lie in the same fibre over the projective line defines a
divisor $D_2$ in $M_{g}$. More precisely, a smooth curve $C$ of genus $g=2k$
defines a point of $D_2$ if it admits a degree $k+1$ map to ${\PP}^1$
with simple ramification points and $6k-1$ branch points.
Similarly, the condition that two
ramification points collide (and then define a triple ramification
point) defines a divisor $D_3$ in $M_{g}$. Their closures give divisors
in $\bM_g$, again denoted by $D_2$ and $D_3$. These divisors
are important and appeared already in the paper \cite{H} of Harris.

Harris calculated the class of $D_3$ in \cite{H} in 1984, but the
class of $D_2$ escaped determination so far. By using a recent result of
Kokotov, Korotkin and Zograf \cite{KKZ} we are now able to calculate 
this class. 
Besides the calculation of the class
$D_2$ with global tools (i.e.\ without the use of test curves),
another purpose of our paper is to study some aspects of the
geometry of the map $\overline{H}_{d,g} \to \overline{M}_g$.

In order to formulate the result we recall that the Picard group with rational 
coefficients of the Deligne-Mumford
stack $\bM_g$ is generated by
the class $\lambda$ of the Hodge bundle and the classes $\delta_j$ of the
boundary divisors $\Delta_j$ for $j=1,\ldots,[g/2]$.

Our result reads as follows.

\begin{theorem}
\label{mainTh}
Let $g=2k$ be an even natural number. The class of $D_2$ on $\bM_{g}$ 
can be written as $c_{\lambda}\lambda + \sum_{j=0}^k c_j \delta_j$
with the coefficients $c_{\lambda}$ and $c_j$ given by
$$
c_{\lambda}= 6\, N \frac{6k-1}{2k-1} (k-2)(k+3),
$$
and
$$
c_0= -\frac{2N}{2k-1}(k-2)(3k^2+4k-1),
$$
and for $1\leq j \leq k$
$$
c_j= -  3N \, \frac{j(2k-j)}{2k-1}(6k^2-4k-7)+
\frac{9}{2} \, j (2k-j) \, \alpha(k,j).
$$
Here $N={2k \choose k+1}/k$ and 
$\alpha(k,j)$ is the combinatorial expression
$$
\alpha(k,j)=\frac{j(2k-j)+k}{k(k+1)} {j \choose [j/2]} {2k-j\choose k-[j/2]}
\qquad \text{for $j$ even}
$$
and
$$
\alpha(k,j)=
\frac{(j+1)(2k-j)}{k(k+1)} {j+1 \choose 1+[j/2]}{2k-j-1 \choose k-1-[j/2]}
\qquad \text{for $j$ odd} \, .
$$
\end{theorem}

\end{section}
\begin{section}{The Hurwitz Scheme}
\label{section Hurwitz scheme}
We call a degree $d$ cover $C_1\to C_2$ of Riemann surfaces
simple if every fibre has at least $d-1$ distinct points.
Let $H_{g,d}$ be the Hurwitz scheme of simple covers of
the projective line ${\PP}^1$ of degree $d$ and genus $g$ with 
ordered branch points and $\overline{H}_{g,d}$ the 
compactification of the Hurwitz scheme by the admissible covers 
with an ordering of the branch points, see \cite{HMu}, p.\ 57.
This is an irreducible projective scheme. 
Recall that two admissible covers $f_i : C_i \to P_i$ are considered
equivalent if there exist isomorphisms $h: C_1 \to C_2$
and $\gamma: P_1 \to P_2$ (preserving the markings) with $f_2 \circ h= 
\gamma \circ f_1$.

In this paper we restrict to the case of even genus $g=2k$ 
and degree $d=k+1$.
Then the Brill-Noether number of linear systems of projective
dimension $r=1$ and degree $d$ equals $\rho= g-(r+1)(g+r-d)=0$.
By the Hurwitz-Zeuthen formula the number of (simple) branch points
is $b=6k$ and the dimension $3g-3$ of the Hurwitz scheme equals that of
$\bM_g$.

There is a natural map
$\pi : \overline{H}_{g,d} \to \bM_g$ with $\bM_g$ the moduli space of stable
curves of genus $g$, defined by contracting the unstable rational
components of an admissible cover.
Moreover, there is also a natural map $q$ of $\overline{H}_{g,d}$
to the moduli space $\bM_{0,b}$ of stable curves of genus $0$
with $b$ marked points. The Hurwitz space thus forms a 
correspondence between ${\bM}_{2k}$ and ${\bM}_{0,6k}$:
\begin{displaymath}
\begin{xy}
\xymatrix{
\bH_{g,d} \ar[r]^{q} \ar[d]^{\pi}& {\bM_{0,6k}} \\
\bM_{g} \\
}
\end{xy}
\end{displaymath}
For a general curve $C$ of genus $g=2k$ the number of $g^1_d$'s 
with $d=k+1$ equals $N={2k \choose k+1}/k$, 
and the natural map $\pi: \overline{H}_{g,d}\to\bM_g$
is generically finite of degree $(6k)! \, N$.

The boundary $\bH_{g,d}-H_{g,d}$ consists of a finite number of
divisors. 
An irreducible divisor in the boundary of $\bH_{g,d}$ maps under $q$ to
an irreducible divisor in the boundary of $\bM_{0,b}$. 
The irreducible boundary divisors of
$\bM_{0,b}$ correspond bijectively to 
the decompositions $\{1,\ldots,b\}= \Lg \sqcup \Lg^c$
into two disjoint subsets $\Lg, \Lg^c$, each with at least two elements. 
We shall write $S^{\Lg}$ for such a boundary divisor with the rule 
that $S^{\Lg}=S^{\Lg^c}$.
The generic member of $S^\Lg$ is a stable rational curve with two
irreducible components, ${\PP}_1$ and ${\PP}_2$, meeting in a point $s$
such that the marked points corresponding to $\Lg$ all lie on one 
of ${\PP}_1$ and ${\PP}_2$.

Under the map $\pi:\overline{H}_{g,d}\to\bM_g$ an irreducible boundary 
divisor of $\bH_{g,d}$ either maps to the boundary of 
$\bM_g$ or has a non-empty intersection with $M_g$. 
We first determine the boundary divisors
that map dominantly to an irreducible divisor in the boundary of 
$\bM_g$; in a later section we determine the irreducible components
of these divisors. Recall that the boundary $\bM_g-M_g$ of
$\bM_{g}$ consists of the irreducible divisors $\Delta_j$ with
$0\leq j \leq [g/2]$, where the generic element of $\Delta_0$ is an 
irreducible one-nodal curve and the generic element of $\Delta_j$
is a curve with two irreducible components of genus $j$ and $g-j$ meeting 
in one point.
\end{section}
\begin{section}{Boundary Divisors mapping to the boundary of $M_g$}
\label{section btob}
We determine which divisors in the boundary of $\bH_{g,d}$ with
$g=2k$ and $d=k+1$ map 
dominantly to an irreducible boundary divisor of $\bM_g$.

\begin{proposition}\label{description}
Let $0 \leq j \leq k$. There are $[j/2]+1$ boundary
divisors $E_{j,c}$ with $c=0,\ldots, [j/2]$ 
mapping dominantly to $\Delta_j$ under $\pi:\overline{H}_{g,d}\to\bM_g$.

\item{\rm i)} For $j\geq 1$ the divisor $E_{j,c}$ decomposes as
$\sum_{\Lg} E_{j,c}^{\Lg}$ with $\Lg$ running over the subsets 
of $\{1,\ldots,6k\}$ of cardinality $3j$, where we identify $\Lg$
with $\Lg^c$ if $j=k$. 
The general element $\varphi: X \to P$ 
of $E_{j,c}^{\Lg}$ maps to a curve $P={\PP}_1 \cup {\PP}_2$
with ${\PP}_1$ (resp.\ ${\PP}_2$) carrying the $6k-3j$  (resp.\ $3j$) 
marked points of $\Lambda^c$ (resp.\ $\Lambda$).
The inverse image of ${\PP}_1$ consists of a smooth curve $C_1$ of 
genus $2k-j$ and $c$ smooth rational curves $R_1,\ldots, R_c$, 
while the inverse image of ${\PP}_2$ consists of a smooth curve 
$C_2$ of genus $j$ and $k-j+c$ smooth rational curves 
$S_1,\ldots ,S_{k-j+c}$.  Each $R_{\mu}$ meets $C_2$ in a single point 
$q_{\mu}$ and each $S_{\nu}$ meets $C_1$ in a single point $p_{\nu}$.
The curves $C_1$ and $C_2$ meet in a single point $p$.

The map $\varphi$ restricted to $C_1$ (resp.\ $C_2$) 
has degree $k+1-c$ (resp.\ $j+1-c$) and has degree $1$ on the 
$R_{\mu}$ and $S_{\nu}$. The $p_{\nu}$ and $q_{\mu}$ are not ramification
points, but $p$ has ramification degree $j+1-2c$ and the points
$q_{\mu}, p_{\nu}, p$ all map to $s$.
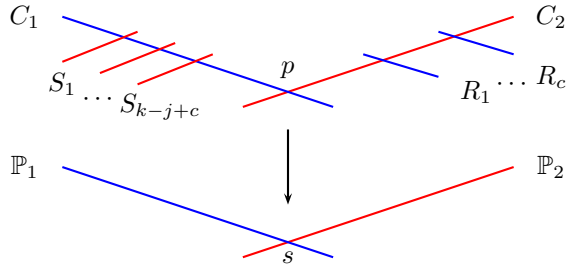
\begin{center}
\begin{pspicture}(-3,-2)(3,3)
\psline[linecolor=red](-0.6,-1.2)(3,0) 
\psline[linecolor=blue](-3,0)(0.6,-1.2) 
\psline[linecolor=red](-0.6,0.8)(3,2) 
\psline[linecolor=blue](-3,2)(0.6,0.8) 
\psline[linecolor=red](-3,1.4)(-2,1.8)
\psline[linecolor=red](-2.5,1.25)(-1.5,1.65)
\psline[linecolor=red](-2,1.1)(-1,1.5)
\psline{->}(0,0.5)(0,-0.5)
\psline[linecolor=blue](2,1.8)(3,1.5) 
\psline[linecolor=blue](1,1.5)(2,1.2)  
\rput(-3.5,2){$C_1$}
\rput(3.5,2){$C_2$}
\rput(-3.5,0){${\PP}_1$}
\rput(3.5,0){${\PP}_2$}
\rput(0,1.3){$p$}
\rput(0,-1.2){$s$}
\rput(-3,1.1){$S_{1}$}
\rput(-2.5,0.9){$\cdots$}
\rput(-1.7,0.8){$S_{k-j+c}$}
\rput(2.5,1){$R_{1}$}
\rput(3,1.1){$\cdots$}
\rput(3.5,1.2){$R_{c}$}
\end{pspicture}
\end{center}


\item{\rm ii)} For $j=0$ the divisor $E_0=E_{0,0}$ decomposes as
$\sum_{\Lg} E_0^{\Lg}$ with $\Lg$ running through the $2$-element subsets of 
$\{1,\ldots,6k\}$. The general element $\varphi: X \to P$ 
of $E_0^{\Lg}$ 
has $P={\PP}_1 \cup {\PP}_2$ with ${\PP}_1$ (resp.\ ${\PP}_2$) 
carrying the $2$ (resp.\ $6k-2$) marked points of 
$\Lambda$ (resp.\ $\Lambda^c$).
The inverse image of ${\PP}_1$ consists of a smooth rational curve
$C_1$ and $k-1$ smooth rational curves $R_1,\ldots,R_{k-1}$
while the inverse image of ${\PP}_2$ is a smooth curve $C_2$
of genus $2k-1$.
The curves $C_1$ and $C_2$ meet in two points $p$ and $q$ and
each $R_{\mu}$ meets $C_2$ in a single point $p_{\mu}$ that is
not a ramification point on $C_2$.

The map $\varphi$ restricted to $C_2$ has degree $k+1$, while restricted
to $C_1$ it has degree $2$ and degree $1$ on the $R_{\mu}$.

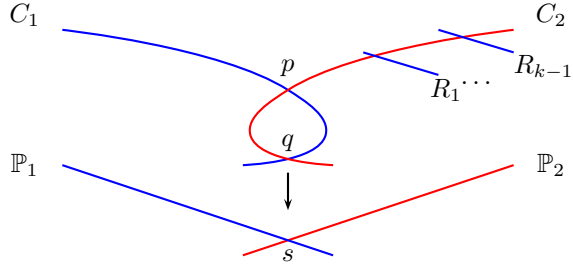
\begin{center}
\begin{pspicture}(-3,-2)(3,3)
\psline[linecolor=red](-0.6,-1.2)(3,0) 
\psline[linecolor=blue](-3,0)(0.6,-1.2) 
\pscurve[linecolor=blue](-3,1.8)(0,1)(0.5,0.4)(-0.6,0) 
\pscurve[linecolor=red](3,1.8)(0,1)(-0.5,0.4)(0.6,0) 
\psline{->}(0,-0.1)(0,-0.6)
\psline[linecolor=blue](2,1.8)(3,1.5)
\psline[linecolor=blue](1,1.5)(2,1.2)
\rput(-3.5,2){$C_1$}
\rput(3.5,2){$C_2$}
\rput(-3.5,0){${\PP}_1$}
\rput(3.5,0){${\PP}_2$}
\rput(0,1.3){$p$}
\rput(0,0.3){$q$}
\rput(0,-1.2){$s$}
\rput(2.5,1.1){$\ldots$}
\rput(2.1,1){$R_1$}
\rput(3.4,1.3){$R_{k-1}$}
\end{pspicture}
\end{center}
\end{proposition}

\begin{remark}
Later we shall prove that the $E_{j,c}^{\Lambda}$  and $E_0^{\Lambda}$
are irreducible.
\end{remark}

We also need to determine the degree of the restriction of $\pi$ 
to the divisor $E_{j,c}$.

\begin{proposition}\label{degree}
\item{i)} For $j>0$ the degree of the restriction 
$\pi_{j,c}:E_{j,c}\to \Delta_j$ of $\pi$ to  $E_{j,c}$ is
$$
\frac{(6k)! \, (j+1-2c)^2}{(j+1)(2k-j+1)}{j+1 \choose c} 
{2k-j+1 \choose k+1-c}.
$$
\item{ii)} For $j=0$ the degree of the restriction $\pi_0: E_0\to \Delta_0$
of $\pi$ to $E_0$ is
$$
\frac{(6k)!}{2k} {2k \choose k+1}= \frac{(6k)!}{2}\, N\,.
$$ 
\end{proposition}

\begin{proof} We shall prove the two propositions \ref{description} and 
\ref{degree} at the same time.

\item{i)} Suppose $j \geq 1$. We first show that the above  loci 
$E_{j,c}$ in $\bH_{g,d}$ defined by describing their general element 
are divisors in the boundary of 
$\bH_{g,d}$. We apply \cite{HMu}, Theorem A, p.\ 71 
[note that there is a misprint in the formulas (*) there: instead of 
$h^0(L(-2d-g-1)p) \geq 1$ one should read: $h^0(L[-(2d-g-1)]p) \geq 1$] 
with $g=j$ (resp.\ $g=2k-j$) and  $d=d_2=j+1-c$ (resp.\ $d=d_1=k+1-c$).
For $g=j$ we have in the notation of loc.\ cit.\ ${\rm min}\,d=j/2+1$ and
${\rm max}\,d=j+1$. Similarly for $g=2k-j$ we have
${\rm min}\,d= (2k-j)/2+1$ and ${\rm max}\,d=k+1\leq 2k-j+1$. 
Hence the range of $d$ satisfies the requirements of the theorem.  
Observe  also that $2d-g-1=j+1-2c$.  
The theorem then implies that the generic pair $(C_2,p)$ with
$C_2$ of genus $j$ and $p$ a point of $C_2$ can be expressed 
in
$$
a(j,d_2)=\frac{j+1-2c}{j+1}\;{j+1 \choose c} 
$$
ways
as a ramified cover of ${\PP}^1$ of degree $d_2$
with all branch points simple except the image of $p$ over which  $p$
is the only ramification point with degree $j+1-2c$.
Similarly, the generic pair $(C_1,p)$ with $C_1$ of genus $2k-j$
and $p \in C_1$ can be expressed in
$$
a(2k-j, d_1)= \frac{j+1-2c}{2k-j+1}\; {2k-j+1 \choose k+1-c}
$$ 
ways as a ramified cover of ${\PP}^1$ of degree $d_1$
with all branch points simple except the image of $p$ over which $p$
is the only ramification point with degree $j+1-2c$.
By a dimension count we have now that the 
locus $E_{j,c}$ is pure of codimension $1$ in 
$\bH_{g,d}$ and hence $E_{j,c}$ 
defines a divisor. 

The degree of the restricted map $\pi_{j,c}: E_{j,c} \to \Delta _j$  
is given by 
$$
(6k)!\, a(j, d_2)\, a(2k-j, d_1)
$$ 
and this equals
$$
(6k)!\,  \frac{j+1-2c}{j+1}\;{j+1 \choose c} \, 
\frac{j+1-2c}{2k-j+1}\; {2k-j+1 \choose k+1-c}\,.
$$
But by the identity
$$
\sum_{c=0}^{[j/2]}\frac{(j+1-2c)^2}{(j+1)(2k-j+1)} 
{j+1 \choose c}{2k-j+1 \choose k+1-c}=\frac{1}{k} {2k \choose k+1}
$$
we have $\sum_{c=0}^{[j/2]} \deg \pi_{j,c} =\deg \pi=N$
and since $\bH_{g,d}$
is projective and irreducible and $\overline{M}_g$ 
is projective, irreducible and smooth in codimension $2$, 
there is no room for other divisors 
in the boundary mapping dominantly to $\Delta_j$.

\item{ii)} 
For $j=0$ the analysis gives that the curve described in the Proposition
\ref{description} ii)
is a general member of a divisor which maps to $\Delta_0$. 
Indeed, in this situation  $\rho =1$ and hence the curve $C_2$ possesses 
a $g^1_{k+1}$ passing through two generic points: 
the pre-image of the space $W^1_{k+1}$ in ${\rm Sym}^{k+1}C_2$ is 
$2$-dimensional and hence intersects $p+q+{\rm Sym}^{k-1}C_2$ 
(which of class $x^2$, where $x$ is the ample class representing 
the divisor $p+{\rm Sym}^{k}C_2$, see \cite{ACGH} Ch.\ VII, Prop.\ 2.2)  
for every choice of $p,q$.
The maps $C_i \to {\PP}_i$ 
are not ramified at the points $p,q$. Indeed, by the above 
mentioned Theorem A in \cite{HMu}, a generic couple $(C_2,p)$, with 
$g(C_2)=2k-1$, possesses a finite number of pencils $\gamma$ 
of degree $k+1$ with  
$\gamma \geq 2p$. Therefore, for a generic $q$ there is no such 
pencil with $\gamma \geq 2p+q$. By \cite{GH}, 
Main Theorem 2c, p.\ 235, there are $\frac{1}{k}{2k\choose k+1}=N$ 
distinct such linear systems.  

Note that the symmetric group ${\mathbb S}_{6k}$ 
does not act freely on $E_0$,  since we can compose $\pi$ with 
the automorphism of
$P=\PP_1 \cup \PP_2$ that is the identity on $\PP_2$, 
fixes the intersection point $s$ and exchanges the 
two branch points on ${\PP}_1$. This lifts to an automorphism of
$C_1$ fixing $p$ and $q$ and interchanging the ramification points.
Therefore the degree of the 
restricted map $\pi_0: E_0 \to \Delta_0$ is 
$\frac{(6k)!}{2}\,N$,   which is half of the   
generic degree of the map  $\pi: \overline{H}_{d,g} \to \bM_g$. 
On the other hand, a local analysis shows, 
see  \cite{HMu}, bottom of p.\ 76, that the map $\pi$ is 
simply ramified along the divisor $E_0$.
This shows that $E_0$ is a divisor in the boundary
which maps dominantly to $\Delta_0$ and there is no room for other 
divisors.
\end{proof}
\end{section}
\begin{section}{Boundary divisors not mapping to the boundary of $M_g$}
\label{section btos}
We now determine the divisors in the boundary $\bH_{g,d}$ that map dominantly
to a divisor in $\bM_g$ that hits $M_g$.

\begin{proposition}\label{description2}
There are two divisors  $E_2$ and $E_3$ in the boundary of $\bH_{g,d}$
that under $\pi$ map dominantly 
to a divisor in $\bM_g$ that has non-zero intersection with $M_g$.
We have a decomposition 
$ E_2= \sum_{\Lg,\alpha} E_2^{\Lg,\alpha}$ 
into $2 {6k \choose 2}$ divisors
with  two possibilities for $\alpha$ and 
$\Lg \subset \{1,\ldots,6k\}$ and $\# {\Lg} =2$. 
Similarly, we have a decomposition $E_3=\sum_{\Lg} E_3^{\Lg}$
in ${6k \choose 2}$ divisors.
Their description is as follows.

\item{i)} Each general member $\varphi:X\to P$ of $E_2^{\Lg,\alpha}$ maps to a 
curve $P={\PP}_1 \cup {\PP}_2$ with ${\PP}_1$ (resp.\ ${\PP}_2$)
carrying the $2$ (resp.\ $6k-2$) marked points of $\Lambda$ 
(resp.\ $\Lambda^c$). The inverse image of
${\PP}_2$ is a smooth curve $C$ of genus $g$ mapping with degree $k+1$
to ${\PP}_2$, while the inverse image
of ${\PP}_1$ consists of $k-3$ smooth rational curves $R_1,\ldots,R_{k-3}$
mapping with degree $1$ to ${\PP}_1$ and two smooth rational curves
$S_1,S_2$ mapping with degree $2$ to ${\PP}_1$. 
The intersection points $q_i$ of $S_i$ with $C$ are ramification points
and $\alpha$ is a marking of the $q_i$.
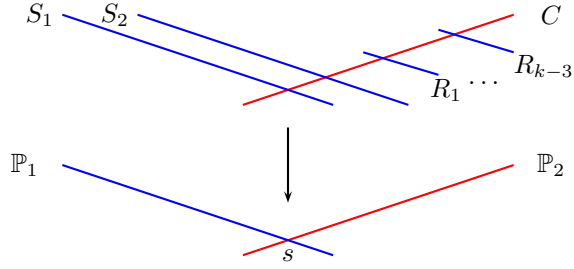
\begin{center}
\begin{pspicture}(-3,-2)(3,3)
\psline[linecolor=red](-0.6,-1.2)(3,0) 
\psline[linecolor=blue](-3,0)(0.6,-1.2) 
\psline[linecolor=red](-0.6,0.8)(3,2) 
\psline[linecolor=blue](-3,2)(0.6,0.8)
\psline[linecolor=blue](-2,2)(1.6,0.8)
\psline{->}(0,0.5)(0,-0.5)
\psline[linecolor=blue](2,1.8)(3,1.5)
\psline[linecolor=blue](1,1.5)(2,1.2)
\rput(-3.3,2){$S_1$}
\rput(-2.3,2){$S_2$}
\rput(3.5,2){$C$}
\rput(-3.5,0){${\PP}_1$}
\rput(3.5,0){${\PP}_2$}
\rput(0,-1.2){$s$}
\rput(2.6,1.1){$\ldots$}
\rput(2.1,1){$R_1$}
\rput(3.4,1.3){$R_{k-3}$}
\end{pspicture}
\end{center}

\item{ii)} Each general member $\varphi:X\to P$ of $E_3^{\Lg}$ maps to a
curve $P={\PP}_1 \cup {\PP}_2$ with ${\PP}_1$ (resp.\ ${\PP}^2$)
carrying the $2$ (resp.\ $6k-2$) marked
points of $\Lambda$ (resp.\ $\Lambda^c$). The inverse image of
${\PP}_2$ is a smooth curve $C$ of genus $g$ mapping with degree $k+1$
to ${\PP}_2$, while the inverse image
of ${\PP}_1$ consists of $k-2$ smooth rational curves $R_1,\ldots,R_{k-2}$
mapping with degree $1$ to ${\PP}_1$ and one smooth rational curve
$S$ mapping with degree~$3$ to ${\PP}_1$.
The intersection point $q$ of $S$ with $C$ is a ramification point 
of degree $3$, while the intersections of $C$ with the $R_{\nu}$
are not ramification points on $C$.

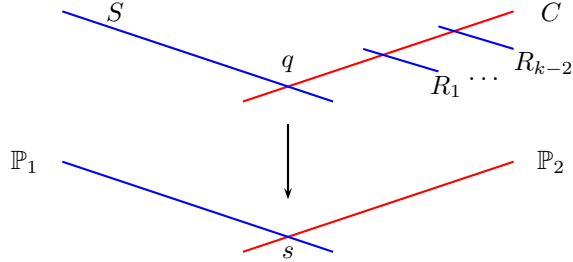
\begin{center}
\begin{pspicture}(-3,-2)(3,3)
\psline[linecolor=red](-0.6,-1.2)(3,0) 
\psline[linecolor=blue](-3,0)(0.6,-1.2) 
\psline[linecolor=red](-0.6,0.8)(3,2) 
\psline[linecolor=blue](-3,2)(0.6,0.8)
\psline{->}(0,0.5)(0,-0.5)
\psline[linecolor=blue](2,1.8)(3,1.5)
\psline[linecolor=blue](1,1.5)(2,1.2)
\rput(-2.3,2){$S$}
\rput(3.5,2){$C$}
\rput(-3.5,0){${\PP}_1$}
\rput(3.5,0){${\PP}_2$}
\rput(0,1.3){$q$}
\rput(0,-1.2){$s$}
\rput(2.6,1.1){$\ldots$}
\rput(2.1,1){$R_1$}
\rput(3.4,1.3){$R_{k-2}$}
\end{pspicture}
\end{center}
\end{proposition}
\begin{remark}
Later we shall prove that the divisors $E_2^{\Lambda, \alpha}$ and 
$E_3^{\Lambda}$
are irreducible divisors.
\end{remark}

\begin{proof}

If we want the image of an admissible cover to be a smooth curve of 
genus $g$ we must have over (say) $\PP_2$ a smooth curve of genus $g=2k$ 
and no rational components. Indeed, otherwise the restriction of the 
covering map on $C$ has degree $\leq k=d-1$. 
But then $\rho \leq -2$ and hence the image cannot be a divisor, 
see \cite{EH2}, Thm.\ 1.1. Over $\PP_1$ we then have only 
rational curves. A naive dimension count shows that the number of 
branch points on $\PP_2$  outside the singular point should be $b-2$ 
and hence on $\PP_1$ there should be $2$. In fact, in this case the total 
number of branch points on $\PP_2$ is $b-1$ and hence the number 
of parameters for the curve $C$ is $b-1-3=6k-4=3g-4$ as required
(and this is the only case where this happens). 
Then, over $\PP_1$  only two cases are possible, namely the ones 
described  in the statement of the proposition, see also \cite{HMo}, 
p.\ 181-83 and Figures 3.146 on p.\ 177, and 3.154 on p.\ 183 
(the first case corresponds to the situation where two branch points 
come together but the two ramification points remain distinct points 
on the same fiber and the second to the case 
where the two ramification points come together too). Note that
in the first case each of the $S_i$ contain one marked point
not mapping to $s$. This gives the marking $\alpha$.
\end{proof}
\begin{remark}\label{hurwitznumber}
For later use we notice that the Hurwitz number of degree $3$ covers
of ${\PP}^1$ of genus $0$ with one triple ramification point and 
two simple branch points is~$1$. The involution on ${\PP}^1$ fixing
the triple branch point and interchanging the other two branch points
lifts to the cover. Similarly, the Hurwitz number of degree $2$ covers
of genus $0$ with two branch points is $1$ and the involution
interchanging the two branch points and fixing a third point lifts
to the cover.
\end{remark}
\begin{lemma}\label{deformations} 
The formal local ring that pro-represents the infinitesimal deformations
of a general point of $E_3^{\Lambda}$ is smooth, but for a general
point of $E_2^{\Lg,\alpha}$ it equals
$$
{\CC}[[t_{11},t_{12},t_2,\ldots,t_{b-3}]]/\langle t_{11}^2-t_{12}^2\rangle
\, .
$$
The corresponding point on the coarse moduli space 
$\overline{H}_{g,d}$ is smooth.
\end{lemma}
\begin{proof} The result for $E_3^{\Lambda}$ follows from \cite{HMu}, p.\ 62.
For each general point of $E_2^{\Lg, \alpha}$  
the complete local ring pro-representing the infinitesimal deformations
equals
$$
{\CC}[[t_1,\ldots,t_{b-3},t_{11},t_{12}]]/
\langle t^2_{11}-t_1, \, t^2_{12}-t_1\rangle \cong 
{\CC}[[t_{11}, t_{12},t_2,\ldots,t_{b-3} ]]/
\langle t^2_{11}-t^2_{12}\rangle \; .
$$
Indeed, for a cover $C \to P$ locally over $s={\PP}_1 \cap {\PP}_2$ 
the equations near $q_i$ at the two rational tails $S_i$
are $x_iy_i=t_{1i}$ for $i=1,2$. The involutions on these tails are
given by $y_i\mapsto -y_i$ and this induces
 $t_{11} \mapsto -t_{11}$ and $t_{12} \mapsto -t_{12}$. 
Hence the quotient is the smooth ring
$ {\CC}[[u, t_2,\ldots,t_{b-3} ]] $
and it is the complete local ring of the coarse moduli space 
$\overline{H}_{g,d}$ at a general point of $E_2^{\Lambda,\alpha}$.
\end{proof}
\end{section}
\begin{section}{The Trace Curve of a Pencil}\label{trace_curves}
In this section we collect a few results about trace curves
that we need in the sequel. 
Let $C$ be smooth curve of genus $g$ together with a pencil (a 
linear system of projective dimension $1$) of degree $d$, say $\gamma$. 
We define the {\sl trace curve} of $\gamma$ by
$$
T_{\gamma} =\{ (p,q) \in C \times C \, : \gamma \geq p+q \}.
$$
Here by $\gamma \geq p+q$ we mean that there is an effective divisor in 
$\gamma$ containing $p$ and $q$.
The following lemma gives information on the singularities that $T_{\gamma}$
might have. In the following we shall assume that our pencils are base point
free.
\begin{lemma}
If $\gamma$ is base-point free then $T_{\gamma}$ is smooth except for
possible singularities at points $(p,q)$ with both $p$ and $q$ ramification
points of $\gamma$. A point $(p,p)$  with $p$ a ramification point  
of order $m$ (of the map to ${\PP}^1$)
gives an ordinary singularity of order $m-1$ on $T_{\gamma}$.
Morover, if $(p,q) \in T_{\gamma}$
with $p\neq q$ and $p$ and $q$ simple ramification points
then the singularity of $T_{\gamma}$ at $(p,q)$ is a simple node.
\end{lemma}
\begin{proof}
Let $\{ f, g\}$ be a basis of the pencil $\gamma$ and let $(p,p)$
be a point of $T_{\gamma}$ and let $z$ be a local coordinate at $p$.
Then $T_{\gamma}$ is locally at $(p,p)$ given by $h=0$ with
$$
h(z_1,z_2)= \frac{f(z_1)g(z_2)-f(z_2)g(z_1)}{z_1-z_2} \, .
$$
We may assume that ${\rm ord}_p(f)=m>0$ and ${\rm ord}_p(g)=0$. 
Write $f=z^m f_1$ and find in the local ring
$$
h(z_1,z_2)= \frac{z_1^{m}-z_2^{m}}{z_1-z_2} f_1(0)g(0),
$$
so locally at $(p,p)$ the curve $T_{\gamma}$ consists of $m-1$ 
branches passing transversally through $(p,p)$.

If $(p,q)\in T_{\gamma}$ with $p\neq q$ and $z$ (resp.\ $w$) a
local coordinate at $p$ (resp.\ $q$) we write $f=f_1(z)$ and $f=f_2(w)$
and similarly $g=g_1(z)$ and $g=g_2(w)$ in the local rings of
$p$ and $q$. The equation of $T_{\gamma}$ is then
$h(z,w)= f_1(z)g_2(w)-f_2(w)g_1(z)=0$.
Write $f_1= a_0 +a_1 z+ \ldots$ and $g_1=b_0+b_1 z+ \ldots$;
furthermore $f_2=c_0+c_1 w+ \ldots $ and $g_2=d_0+d_1w+\ldots$
and find that a singularity at $(p,q)$ 
means that besides $a_0d_0- b_0c_0=0$ we
have $a_1d_0-b_1c_0=0$ and $a_0d_1-b_0c_1=0$. We may assume that
$a_0=0$ and $b_0\neq 0$, hence $c_0=0$ and $d_0\neq 0$, so that
a singularity means $a_1=c_1=0$, i.e.\ both $p$ and $q$ are ramification
points. Then the next term in $h$ is $a_2d_0z^2-c_2b_0w^2$
and this shows that if $a_2$ and $c_2$ do not vanish we have a simple
node.
\end{proof}
\begin{lemma}\label{trace_curve_irr}
Let $\gamma$ be a base point free $g^1_d$ with all branch points simple
except one with arbitrary ramification. Then $T_{\gamma}$ is irreducible.
\end{lemma}
\begin{proof}
Consider the map $T_{\gamma} \to {\PP}^1$ defined as the composition of
the first projection $T_{\gamma} \to C$ composed with the map
$C \to {\PP}^1$ defined by $\gamma$. All singular points and all
the ramification points of $T_{\gamma}\to {\PP}^1$ lie over the 
branch points of $C\to {\PP}^1$. So for both coverings we consider 
the same monodromy group ($\pi_1$ of the punctured line).
Since $T_{\gamma} \subset C\times C$ the monodromy action for 
$T_{\gamma}$ is induced by the monodromy action for $C$. 
By showing that the latter is doubly transitive the result will follow.
Since $C\to {\PP}^1$ is simply branched except possibly at one point
the monodromy is generated by the transpositions at the simple branch points
(since the product of the permutations of all branch points is $1$).
We thus see that this generates a transitive subgroup of ${\mathbb S}_{d}$,
hence it is the whole symmetric group and therefore doubly transitive.
\end{proof}

\begin{corollary}
\label{trace_curves_singularities}
The trace curves induced by the pencils on $C_1,C_2$ as in i) of
Proposition \ref{description} and on $C$ as in ii) of  
Proposition \ref{description2} have one singular point which is 
an ordinary singularity and lies on the diagonal.
The trace curve induced by the pencil on  $C_2$  as in ii) of  
Proposition \ref{description} is smooth.
The trace curve induced by the pencil on  $C$ as in i) of  
Proposition \ref{description2} has two nodal singularities 
at two symmetric points. Moreover all the above trace curves 
are irreducible.
\end{corollary}

\end{section}
\begin{section}{An Irreducibility Result}
In this section we shall prove that the boundary divisors $E_{j,c}^{\Lg}$
defined in Proposition \ref{description} 
and the boundary divisors $E_2^{\Lambda,\alpha}$, $E_3^{\Lg}$ 
defined in Proposition \ref{description2} are irreducible. 
We start with the result for $E_2$ and $E_3$.

Let $\underline{c}$ be a conjugacy class of the symmetric group 
${\mathbb S}_d$ on $d$ objects. It is given by a partition  of $d$. 
We consider the Hurwitz space ${\mathcal H}_{d, b, \underline{c}}$ 
parametrizing isomorphism classes of (connected) Riemann surfaces 
that are degree $d$ covers of ${\PP} ^1$ that are simply branched at $b$ 
(unordered) points of the projective line different from 
infinity and have ramification type $\underline{c}$ over infinity. 
This has the structure of a smooth analytic space; this may be proved
as in \cite{F}.

We define $\Pi_b:=(\CC ^1)^{\times b}-\Delta$ with $\Delta$ the big diagonal
and $\Sigma_b:={\rm Sym}^b \CC - D$, with $D$ the discriminant locus
and then have a natural map  $p: \Pi_b \to \Sigma_b$.

There is a natural covering map $\mu : {\mathcal H}_{d, b, \underline{c}} 
\to \Sigma_b$ by assigning
 to each point of  ${\mathcal H}_{d, b, \underline{c}} $ the set of
$b$ points with simple branching. We get a projection map
$$
{\rm pr}_2: 
{\mathcal H}_{d, b, \underline{c}} \times _
{\Sigma _b} \Pi_b  \to  \Pi_b. 
$$ 

Let now $H_{d, b, \underline{c}}$ be the Hurwitz space
parametrizing isomorphism classes of (connected) Riemann surfaces 
that are degree $d$ covers of $\PP ^1$ simply branched at $b$ 
{\em ordered} points of the projective line and have an extra point 
with ramification of type $\underline{c}$, modulo the equivalence 
relation that two such covers $f_i: C_i \to \PP ^1$ are equivalent 
if there exist isomorphisms $h:C_1\to C_2$ and $\gamma :\PP^1 \to \PP^1$ 
with $f_2 \circ h=\gamma \circ f_1$. 

There is a natural surjective map 
$$
m: {\mathcal H}_{d, b, \underline{c}} \times _{\Sigma _b} \Pi_b  
\to H_{d, b, \underline{c}}, \eqno(*)
$$
given by associating to the cover $f:C \to {\PP}^1$ and a set of ordered branch
points $\{a_1,\ldots,a_b\}$ the cover with its ordered branch points.

\begin{theorem}
With the notations as before and ${\underline{c}}$ the conjugacy class of 
$\phi=(12)(34)$ or of $\phi=(123)$ the Hurwitz space 
$H_{d, b, \underline{c}}$ is irreducible.
\end{theorem}
\begin{corollary}
The divisors $E_2^{\Lambda,\alpha}$ and $E_3^{\Lg}$ are irreducible.
\end{corollary}
We first deduce the corollary from the theorem.
With $\underline{c}$ the type of a $3$-cycle, say $(123)$, 
and with $b=6k-2$  we have a natural inclusion 
$\nu: H_{d, b, \underline{c}}  \to E_3^{\Lg}$ given as
follows. A point of $H_{d, b, \underline{c}}$ corresponds to a 
cover $C_2 \to \PP_2$ of the projective line  with an ordering 
of the $6k-2$ branch points (which we assume to be indexed by the set $\Lg^c$). 
Then $\nu $ sends this point to the point of $E^{\Lg}_3$
corresponding to the admissible cover $X\to P=\PP_1 \cup \PP_2$, 
with  $\PP_1$ containing the marked points $p_i$ with $i \in \Lg$ and $X$ 
the curve with $C_2$ over $\PP_2$, while over $\PP_1$ 
we have a union of rational curves attached at the ramification points 
of $C_2$ over infinity with the appropriate ramification conditions. 
Note that the positions of the two points $p_i$, $i\in \Lg$, on ${\PP}_1$ do not matter
because of the automorphism group of ${\PP}^1$, cf.\ also 
Remark \ref{hurwitznumber}.
This is a dominant map since its image contains 
the general member  of $E^{\Lg}_3$.  
Since $H_{d, b, \underline{c}}$ 
is irreducible we conclude that  $E_3^\Lg$ is irreducible. 
Similarly for $E_2^{\Lambda,\alpha}$, but here we have to take into account
a marking of the two ramification points of degree $2$ on $C_2$ 
lying over the same point.

\begin{proof}
We prove the theorem by showing that the monodromy of 
${\rm pr}_2$ acts transitively on the fibres
and this implies that
the fibre product
${\mathcal H}_{d, b, \underline{c}} \times _{\Sigma _b} \Pi_b$ is 
connected and by the smoothness it is then irreducible and therefore its 
image $H_{d, b, \underline{c}}$ is irreducible too.  

We choose a point $A \in \Sigma_b$ and a point of $\Pi_b$ mapping to $A$ 
under $p$. That is, we order the points of $A$, say $A=\{a_1,\ldots,a_b\}$.
The points of the fiber $\mu^{-1}(A)$ correspond to the 
${\mathbb S}_d$-conjugacy classes of $b$-tuples $[t_1,\ldots,t_b]$ with
$t_i$ a transposition in the symmetric group
${\mathbb S}_d$ such that these generate ${\mathbb S}_d$ and such that the product $t_1\cdots t_b$ has type ${\underline{c}}$.

By fixing a permutation $\phi$ from the conjugacy class ${\underline{c}}$
we can then describe the fibre $\mu^{-1}(A)$ as the quotient
$$
\Xi_{\phi}^{d,b}/G_{\phi} \, ,
$$
where $G_{\phi}\subset {\mathbb S}_d$ 
is the stabilizer of $\phi$ under conjugation and
$\Xi^{d,b}_{\phi}$ is the set
$$
\Xi ^{d,b}_{\phi} = \{  [t_1,\ldots ,t_b], \;  t_i \mbox{ 
are transpositions generating } {\mathbb S}_d, \;  t_1\cdots t_b=\phi\}
$$
on which $G_{\phi}$ acts by conjugation.

According to \cite{K}, Theorem 1, the braid group $B_b=\pi_1(\Sigma _b, A)$ 
acts transitively on $\Xi ^{d,b}_{\phi}$.  
We consider now the two cases, $\phi=(123)$ and $\phi=(12)(34)$ and we 
prove that in both cases the pure braid group $P_b=\pi_1(\Pi_b,\{a_1,\ldots,a_b\})$ acts transitively on $\Xi ^{d,b}_{\phi}$. 
Note that $B_b/P_b\cong {\mathbb S}_b$.

We work as in \cite{D}, proof of Lemma 3.2. 
We denote by $\Gamma _i$, $i=1,\ldots,b-1$, the standard generators of the 
braid group $B_d$. The action of $\Gamma_i$ on  $\Xi ^{d,b}_{\phi}$ sends
$[t_1,\ldots, t_i,t_{i+1},\ldots ,t_b]$  to 
$[t_1,\ldots, t_{i+1},t_{i+1}t_it_{i+1},\ldots ,t_b]$. 
Moreover, $\Gamma_i$ interchanges the points $a_i$ and $a_{i+1}$.  
We examine now separately the two cases:

{\sl Case i}: $\phi=(123)$. 
We start with the element  $[t_1,\ldots ,t_b]  \in \Xi ^{d,b}_{\phi}$. 
By the above transitivity result we can find an  element 
$\Gamma $ of $B_b$ which sends  $[t_1,\ldots ,t_b]$ to the following 
element of $\Xi ^{d,b}_{\phi}$:
$$ 
\sigma_0=[(13),(12), (14),(14),\ldots, (1d-1),(1d-1),(1d), \ldots,(1d)]\, ,
$$
where the last transposition $(1d)$ appears $b-2(d-3)$ times 
(which by the Hurwitz-Zeuthen formula is an even number).

We now consider the elements $\Gamma _i^3$, $i=1,\ldots,b-1$. 
The action by such an element interchanges $a_i$ and $a_{i+1}$ 
and so the above set of elements acts transitively on the permutation group 
${\mathbb S}_b$ of the indices. On the other hand we observe that it 
acts trivially on  $\sigma_0$, because the supports of two consecutive 
transpositions in $\sigma _0$ have a common part: 
if $[(mn),(kl)]$ denote the $i$th and $(i+1)$th element in $\sigma_0$, then  
if $(mn)=(kl)$ the action of $\Gamma _i$ is trivial, and if $n=k$ but 
$n \neq m\neq l$ then 
the action of $\Gamma_i^3$ is given by
$$ 
[(mn),(nl)]\to [(nl), (ml)] \to [(ml),(mn)] \to [(mn), (nl)]\,. 
$$
Because of the transitivity of the action of the set 
$\Gamma _i^3$, $i=1,\ldots,b-1$, on ${\mathbb S}_b$ 
we may compose $\Gamma$ with an appropriate sequence of the elements 
$\Gamma_i^3$ so that the composition  belongs to the pure braid group $P_b$ 
and the action still sends  our $b$-tuple $[t_1,\ldots ,t_b] $ 
to the fixed element $\sigma_0$. 
This proves that $P_b$ acts transitively on 
$\Xi ^{d,b}_{\phi}$, with $\phi=(123)$. 

{\sl Case ii)}: $\phi=(12)(34)$. 
We work as before with
$$
\sigma_0=
[(12), (13),(13),(34), (14),(14),\ldots, (1d-1),(1d-1),(1d), \ldots,(1d)]\, ,
$$
where the last transposition $(1d)$ appears $b-2(d-2)$ times 
(which is an even number).

This proves that in the two cases the 
product 
${\mathcal H}_{d, b, \underline{c}} \times _{\Sigma _b} \Pi_b$ is
connected. 
\end{proof}

\begin{proposition} Each divisor $E_{j,c}^{\Lg}$ as in 
Proposition \ref{description} is irreducible.
\end{proposition}

\begin{proof}
The irreducibility of the divisors $E^{\Lg}_{j,c}$ is proved in a way 
similar to the case $E_3$. With $d_1=k+1-c$, $b_1=6k-3j$ and 
$\phi= (1 \, 2\ldots j+1-2c)$ we define $\sigma_0=[t_1,\ldots,t_{6k-3j}]$ 
by taking $t_{\nu}=(1, j+2-2c-\nu)$ for $\nu=1,\ldots,j-2c$,
and $t_{j-2c+2\mu+1}=t_{j-2c+2\mu+2}=(1, j+2-2c+\mu)$ 
for $\mu=0,\ldots,k-j+c-1$
and the remaining $t_{\nu}$ are equal to $(1\, 2)$, i.e.\ 
$\sigma_0$ is equal
to
$$
[(1,j+1-2c), \ldots ,(12),\; (1,j+2-2c), (1,j+2-2c),\ldots, 
 (1 d_1), (1 d_1),\; (12),\ldots,(12)].       
$$
Note that $\sigma _0$ contains all the transpositions $(1k),\; k=1, \ldots 
d_1=k+1-c$, hence it generates the symmetric group ${\mathbb S}_{d_1}$. 
The last transposition $(12)$ appears $[6k-3j]-[(j-2c)+2(k-j+c)]$ 
times which is the even number $4k-2j$. 
Hence the product of the transpositions contained in $\sigma _0$  
is $\phi$. 
With $\underline{c}$ the type of $\phi$, an argument similar to 
the case $E_3$ shows that the corresponding Hurwitz space  
$H_{d_1,b_1,\underline{c}}$ with ordered branch points is irreducible. 

Similarly we show that  the Hurwitz space  $H_{d_2,b_2,\underline{c}}$ 
is irreducible with  $d_2=j+1-c$, $b_2=3j$,  and $\underline{c}$ 
the type of $\phi= (12\ldots j+1-2c)$,  by defining
$\sigma_0=[t_1,\ldots,t_{3j}]$ with $t_1,\ldots,t_{j-2c}$ as in the 
preceding paragraph and
$t_{j-2c+2\mu+1}=t_{j-2c+2\mu+2}=(1, j+2-2c+\mu)$ for $\mu=0,\ldots,c-1$
and by setting $t_{\mu}=(1\, 2)$ for the remaining indices.
The last  transposition $(12)$ appears an even number $3j-[(j-2c)+2c]=2j$  
of times.
Therefore the space 
$H_{d_1,b_1,\underline{c}}\times H_{d_2,b_2,\underline{c}}$ 
is irreducible. 

We now define the inclusion  
$\nu: H_{d_1,b_1,\underline{c}}\times 
H_{d_2,b_2,\underline{c}} \to E^{\Lg}_{j,c}$ as follows:
We assume that the  $b_1=6k-3j$ marked points of  
$H_{d_1,b_1,\underline{c}}$ take values in the set $\Lg^c$ and the 
$b_2=3j$  marked points of  $H_{d_2,b_2,\underline{c}}$ take 
values in the set $\Lg$. A point $h_1 \in H_{d_1,b_1,\underline{c}}$ 
(resp.\ $h_2 \in H_{d_2,b_2,\underline{c}}$) 
corresponds to a curve $C_1$ (resp.\ $C_2$) of  genus $2k-j$
(resp. $j$)  with a $g^1_{d_1}$ (resp.\ $g^1_{d_2}$) 
having simple branching except in one fiber which has a point 
$p_1$ (resp.\ $p_2$) of ramification degree $j+1-2c$ and simple 
ramification everywhere else. We then define
$\nu(h_1,h_2)$ to be the admissible cover $X$ 
constructed by the above data as in Proposition \ref{description} i) 
by joining the curves $C_1$ and $C_2$ at the points $p_1$ and $p_2$ 
respectively and attaching rational tails appropriately. 
The map $\nu$  is a dominant map since its image contains  
the general member  of $E^{\Lg}_{j,c}$.  
Since  $H_{d_1,b_1,\underline{c}}\times 
H_{d_2,b_2,\underline{c}}$ is irreducible we conclude 
that  $E^{\Lg}_{j,c}$ is irreducible.   
\end{proof}
  
For the divisor $E_0$ we have a decomposition $E_0=\sum _{\Lg} E_0^{\Lg}$ 
with $\Lg$ running over the subsets of $\{1, \ldots ,6k\}$ with $2$ elements. 
We prove that $E_0^{\Lg}$ is irreducible. 

\begin{proposition}
The divisor $E_0^{\Lg}$ as in Proposition \ref{description} ii) is irreducible.
\end{proposition}
\begin{proof}
Consider the Hurwitz space $H_{k+1,6k-2}$ parametrizing isomorphism 
classes of (connected) Riemann surfaces of genus $2k-1$ that are 
degree $k+1$ covers of $\PP ^1$ simply branched at $6k-2$ ordered 
points of the projective line.  Let  ${\mathcal C}_H$ be the universal 
curve over $H_{k+1,6k-2}$. On ${\mathcal C}_H \times_{H_{k+1,6k-2}}  
{\mathcal C}_H$ we consider the universal trace curve $\Tcal$; 
the fiber  $\Tcal _h$ of $\Tcal $ over a point $h \in  H_{k+1,6k-2}$ 
is the trace curve $\{(x,y)\in C\times C: x+y\leq \gamma \}$ 
corresponding to $\gamma$, the $g^1_{k+1}$ associated to $h$. 
The curve $\Tcal_h$ is an irreducible curve because the $g^1_{k+1}$ 
has simple branching, see section  \ref{trace_curves}. 
Therefore $\Tcal  $ is an irreducible space. We now define a natural 
$2:1$ map $\nu : \Tcal \to E_0^{\Lg}$ as  follows. 
We let the  $6k-2$ branch points of $C$ take values in the set $\Lg^c$. 
A point $h$ of $\Tcal $ corresponds to a curve $C$ with a $g^1_{k+1}$ 
as above, say $\gamma$, 
and a couple $(p,q)$ of points of $C$ with $\gamma \geq p+q$. 
We then define $\nu (h)$ to be the admissible cover as in Proposition \ref{description} ii), with $C_2=C$  and the points $p,q$ as above. 
We attach to $C_2$ the rational curves $C_1$ and 
$R_1, \ldots , R_{k-1}$ as in Proposition \ref{description}. 
As in the case of $E_3$ the position of the branch points on $\PP_1$ 
does not matter. The map $\nu$  is a dominant map since its image 
contains  the general member  of $E_0^{\Lg}$.  Since  $\Tcal $  is 
irreducible we conclude that   $E_0^{\Lg}$ is irreducible.   
\end{proof}

\end{section}
\begin{section}{The Degree of $\pi$ restricted to $E_3$ and $E_2$}\label{sec:degree}
We shall denote the image of the divisor $E_3$ (resp.\  $E_2$) 
under the morphism $\pi: \overline{H}_{2k,k+1} \to \overline{M}_{2k}$ 
by $D_3$ (resp.\  $D_2$). We know that $E_3$  decomposes
as a union of ${6k \choose 2}$ irreducible divisors $E_3^{\Lg}$, 
 with  $\# \Lg=2$ and similarly 
$E_2=\sum_{\Lambda,\alpha} E_2^{\Lambda,\alpha}$
with $2 \, {6k \choose 2}$ components. 
It follows from the results of the preceding section
that the degree of $\pi: E_3^{\Lg} \to D_3$ (resp.\
$\pi: E_2^{\Lg,\alpha} \to D_2$) is the same as the
degree of a map $H_{k+1,6k-2,\underline{c}} \to D_3$
with $\underline{c}$ the type of a $3$-cycle (resp.\ of a cycle of
type $(12)(34)$). In fact, the Hurwitz space 
$H_{k+1,6k-2,\underline{c}}$
can be identified with the Hurwitz space $H_{k+1, 6k-2,3}$ 
(resp. $H_{k+1,6k-2,2+2}$), 
that parametrizes $k+1$ coverings $C \to D$ with $D$ a $6k-1$ 
pointed curve $(D,p_1,\ldots,p_{6k-1})$ of genus $0$ and $C$ 
a connected smooth curve of genus $2k$ which has over $p_1$ 
one point of triple ramification (resp.\ two simple
ramification points) and is simply branched at the points 
$p_2, \ldots, p_{6k-1}$ and unramified everywhere else.
We know that $H_{k+1, 6k-2,3}$ (resp.\ $H_{k+1,6k-2,2+2}$) is
irreducible and hence its compactification 
$\overline{H}_{k+1,6k-2,3}$ (resp.\ $\overline{H}_{k+1,6k-2,2+2}$)
by admissible covers 
(see \cite{D1}, Section 5)   is irreducible.  

\begin{theorem} \label{degreeE}
The degree of $\pi$ restricted to $E_3$ is 
$(6k)!/2$. The degree of $\pi$ restricted to $E_2$
is $(6k)!$.
\end{theorem}
In view of the discussion above it suffices
to prove that the degree of the map
$\overline{H}_{k+1,6k-2,3} \to D_3$ 
(resp.\ $\overline{H}_{k+1,6k-2,2} \to D_2$)
equals $(6k-2)!$,
in other words that the degree is $1$ modulo the action of ${\mathbb S}_{6k-2}$.
Since $\overline{H}_{k+1,6k-2,3}$ is irreducible
it suffices to find an
appropriate smooth point of $D_3$ and determine the degree of 
the fiber over this point.

For this we consider linear systems $g^1_{k+1}$ on a generic curve of
genus $2k-1$ with $6k-4$ simple branch points and one branch point
over which there is one triple ramification point (resp.\ two double
ramification points). We call such a pencil of degree $k+1$ of type $(3)$ 
(resp.\ of type $(2,2)$).

Recall that according to Harris (\cite{H}, Thm 2.1) 
for a general curve $C^{\prime}$ of genus $2k-1$
the number of pencils of degree $k+1$ and of type $(3)$ is finite and equals
$$
b(k)=12 \frac{k-1}{k} {2k \choose k+1} .
$$
Similarly, by the same result (cf.\ loc.\ cit.) for a general $C^{\prime}$
of genus $2k-1$ 
and a general point $p$ on $C^{\prime}$ there are finitely many pencils 
$\gamma$ of degree $k+1$ on $C^{\prime}$ 
with the property that $\gamma \geq 2p$. 
Their number equals
$$
a(k)= \frac{1}{k} {2k \choose k+1} .
$$
Moreover, for a general $C^{\prime}$ of genus $2k-1$ and a general point $p$
the number of pairs $(\gamma,q)$ with $\gamma$ a pencil of degree $k+1$
and $\gamma \geq p+2q$ is finite and equals
$$
c(k)=5 \frac{k-1}{k} {2k \choose k+1}.
$$

\begin{lemma}\label{abclemma}
Let $C'$ be a general curve of genus $2k-1$ and let $p$ be a general point
of $C^{\prime}$. Then there exists a point $q$ on $C^{\prime}$ such that
\begin{enumerate}
\item{} there exists a unique pencil 
$\gamma$ on $C^{\prime}$ of degree $k+1$ 
and type $(3)$ with $\gamma \geq p+q$;
\item{} there does not exist a pencil $\gamma^{\prime}$ on $C^{\prime}$ 
of degree $k+1$ with $\gamma^{\prime}\geq 2p+q$ 
or with $\gamma^{\prime}\geq 
p+2q$.
\end{enumerate}
\end{lemma}
\begin{proof}
Let $\gamma_i$ with $i=1,\ldots,b(k)$ be the type $(3)$ pencils of degree $k+1$
and let $T_3=\cup_{i=1}^{b(k)} T_{\gamma_i}$ be the union of the trace curves
associated to the $\gamma_i$. Note that by Lemma
\ref{trace_curve_irr} each $T_{\gamma_i}$ is irreducible
and contained in $C'\times C'$ and thus $T_3$ has a projection 
$\tau_3 : T_3 \to C'$ to the first factor. 
We now choose a pair $(p,q)$ in $T_3$ which is
sufficiently general; this means that $p$  is not contained
in the image under $\tau_3$ of any multiple point of $T_3$
and $p$ is not contained in a fibre of a $\gamma_i$ containing
a ramification point of $\gamma_i$; in other words $\# \tau_3^{-1}(p)=
k \, b(k)$.

We set $\Sigma_p= \tau_3^{-1}(p)$. The above $p$ is a 
general point on $C^{\prime}$. We consider the
pencils $\gamma_1^{\prime}, \ldots, \gamma_{a(k)}^{\prime}$ of degree $k+1$
on $C^{\prime}$ with $\gamma_i^{\prime} \geq 2p$. Let
$$
\Sigma_p^{\prime}=\{ r^{\prime} \in C^{\prime} : 
\text{ $\gamma^{\prime}_i \geq 2p+r^{\prime}$  
for some $1\leq i \leq a(k)$ } \}.
$$
Then by the result of Harris (\cite{H}, p.\ 44) 
we have $\# \Sigma_p^{\prime}=(k-1) a(k)$ and moreover, if we define
$$
\Sigma_p^{\prime\prime}=\{
r^{\prime\prime} \in C^{\prime} : 
\text{ $\gamma^{\prime}_i \geq p+2r^{\prime\prime}$  
for some $1\leq i \leq c(k)$ } \}
$$
we have by the shape of $a(k)$, $b(k)$ and $c(k)$ that
$ \# \Sigma_p > \# \Sigma_p^{\prime} + \# \Sigma_p^{\prime\prime}$.
Then we can choose a point $q$  in $\Sigma_p -(\Sigma_p^{\prime} \cup
\Sigma_p^{\prime\prime})$ and by taking for $\gamma$ the unique 
$\gamma_i$ such that $(p,q) \in T_{\gamma_i}$
the pencil $\gamma$ and the points $p$ and $q$ satisfy 
the conditions of our lemma.
\end{proof}
Note the similarity of the argument with considerations of Harris
in \cite{H}, p.\ 458.

We now work out the case of $E_3$. After completing that case we give the
modifications in the proof to make it work for $E_2$ too. 

We now take a generic curve $C'$ of genus $2k-1$ and a
pencil of degree $k+1$ of type $(3)$,  
say $\gamma$, on $C'$ and a couple of points $p,q$ as in the lemma.
Then the nodal curve $C= C'/(p \sim q)$ determines
a point $[C]$ of $\overline{M}_g$ with $g=2k$ and this point lies 
on the divisor $\Delta_0$.

\begin{proposition}
The set-theoretic fibre  of the map
$\pi': \overline{H}_{k+1,6k-2,3}/{\mathbb S}_{6k-2} \to D_3$ over the point $[C]$
consists of one point.
\end{proposition}
\begin{proof}
We first describe the admissible cover that represents the unique point
of the fibre. It is the admissible cover $X\to {\PP}_1 \cup {\PP}_2$ 
with ${\PP}_1 \cup {\PP}_2$ the rational curve consisting of two copies
of ${\PP}^1$ intersecting transversally in one point $s$. Over ${\PP}_2$
the curve $X$ has a component $C'$ with a covering $C' \to {\PP}_2$ 
determined by $\gamma$ and the fibre over $s$ contains $p$ and $q$. 
Over ${\PP}_1$ the curve $X$ is the union of a rational curve $R$ which
is a double cover of ${\PP}_1$ intersecting $C'$ at the points $p$ and $q$
with no ramification at these points and having two simple marked branch 
points and $k-1$ rational curves $R_i$ mapping isomorphically to ${\PP}_1$ and
intersecting $C'$ at the remaining $k-1$ points of the fibre over $s$
different from $p$ and $q$.

We now analyze the uniqueness.
The locus of $[C]$ as constructed above has dimension $6k-5$: 
the curve $C'$ is generic of genus $2k-1$ so it contributes $6k-6$ 
to the dimension
and the pair $(p,q)$ is a generic point of the trace curve $T_{\gamma}$ 
in $C'\times C'$ so it contributes~$1$. (Note that $p$ was chosen
general on $C$ and that results in finitely many choices for $q$.)
Since the locus of the admissible covers in $\overline{H}_{k+1, 6k-2,3}$ 
mapping to a rational curve with more than two components has dimension 
$\leq 6k-6$ we conclude that an admissible cover in  
$\overline{H}_{k+1, 6k-2,3}$ mapping to $[C]$ will correspond to a cover 
of a rational curve with exactly two components. 

Such an admissible cover has by definition a single triple ramification 
point over a branch point $p_1$ lying on $\PP_1$ or on $\PP_2$ and not on their
intersection. In order to map to $[C]$, it should contain over  
$\PP_2$ the curve $C'$ and over $\PP_1$ a rational component $R$ 
intersecting $C'$ exactly at the points $p, q$ and 
other rational components $R_j$, each of which intersects $C^{\prime}$
at a unique point $q_j$. 
Since the gonality of the generic curve of genus $2k-1$ is $k+1$
(i.e., the minimum degree of a non-constant map of $C'$ to $\PP_2$), 
there is  no room for other 
rational components over $\PP_2$.

We distinguish two cases: (i) $p_1 \in {\PP}_2$; (ii) $p_1 \in {\PP}_1$.
In the first case, if $p_1 \in {\PP}_2$ then the map $C^{\prime}
\to {\PP}_2$ is of type $(3)$ and by the choice of $p$ and $q$ it coincides
with our $\gamma$. Then we find that $R \to {\PP}_1$ is a $2:1$ covering
and the remaining components $R_j$ map isomorphically to ${\PP}_1$. We thus
retrieve the cover $X$ described in the first paragraph of our proof.

In the second case, if $p_1 \in {\PP}_1$ then $C^{\prime} \to {\PP}_2$
is described by a degree $k+1$ pencil $\gamma^{\prime}$. Then either $R$ or
one of the $R_j$ contains a ramification point of degree~$3$ lying
over $p_1$. If this ramification point lies on $R$ then $\gamma^{\prime}$
has the property that $\gamma^{\prime} \geq 2p+q$  or $\gamma^{\prime}
\geq p+2q$ which is excluded by  lemma \ref{abclemma}. 
If some $R_j$ contains this
ramification point then $q_j$ has ramification degree $\geq 3$ 
which contradicts the generality of $(p,q)$.
\end{proof}

We now do the $E_2$ case which is similar. For this we need the fact
that for a general curve $C^{\prime}$ of genus $2k-1$ the number of
pencils on $C^{\prime}$ of degree $k+1$ and type $(2,2)$ equals
$$
d(k)= 12 \frac{(k-1)(k-2)}{k} {2k \choose k+1}.
$$
This can be calculated as in Harris \cite{H}.
We also need an analogue of lemma \ref{abclemma}.

\begin{lemma}\label{dlemma}
Let $C^{\prime}$ be a general curve of genus 
$2k-1$ and $p$ a general point of $C^{\prime}$.
Then there exists a point $q$ on $C^{\prime}$ such that
\begin{enumerate}
\item{} there exists a unique pencil $\delta$ on $C^{\prime}$ of degree $k+1$
and type $(2,2)$ with $\delta \geq p+q$;
\item{} there does not exist a pair $(\delta^{\prime},q^{\prime})$ with
$\delta'$ a pencil of degree $k+1$ 
and a point $q^{\prime}$ on $C^{\prime}$ with $\delta^{\prime} 
\geq p+q+2q^{\prime}$.
\end{enumerate}
\end{lemma}
\begin{proof}
Let $\delta_j$ with $j=1,\ldots,d(k)$ be the type $(2,2)$ pencils of degree
$k+1$ and let $T_2= \cup_{j=1}^{d(k)} T_{\delta_j}$ be the union of
the trace curves. We now choose a pair $(p,q) \in T_2$ which is sufficiently
general, i.e., $p$ is not contained in the image under the first projection
$\tau_2 : T_2 \to C^{\prime}$ of any multiple point of $T_2$ and $T_3$ (as
defined in lemma \ref{abclemma}) and $p$ is not contained in any fibre
of a $\gamma_i$ (as in lemma \ref{abclemma}) or a $\delta_j$ containing
a ramification point; this gives $\# \tau_2^{-1}(p)= k\, d(k)$.

We now set $S_p=\tau_2^{-1}(p)$. We let $(\delta^{\prime}_j,q_j^{\prime})$
for $j=1,\ldots,c(k)$ be the pairs of pencils $\delta_j^{\prime}$
of degree $k+1$ and points
$q_j^{\prime}$ on $C^{\prime}$ with $\delta_j^{\prime} \geq p+2q_j^{\prime}$.
Now we define
$$
S_p^{\prime}=\{ q^{\prime} \in C^{\prime} : 
\text{ there exists a $j$ such that $\delta_j^{\prime} \geq p+2q_j^{\prime}+q^{\prime}$ } \}.
$$
We have $\# S_p^{\prime}=(k-2)c(k)$ and we see 
using the shape of $c(k)$ and $d(k)$ that $\# S_p > \# S_p^{\prime}$.
We can now choose a point $q$ in $S_p-S_p^{\prime}$ and the unique $\delta_j$
with $1 \leq j \leq d(k)$ such that $(p,q) \in T_{\delta_j}$.
This finishes the proof of the lemma.
\end{proof}

We take a generic curve $C^{\prime}$ of genus $2k-1$ with a pencil $\delta$
of degree $k+1$ and type $(2,2)$ and a couple $(p,q)$ of points as in lemma
\ref{dlemma}. We get a nodal curve $C=C^{\prime}/(p \sim q)$ and a point
$[C]$ on the boundary $\Delta_0$ of $\overline{M}_g$.

\begin{proposition}
The set-theoretic fibre  of the map
$\pi': \overline{H}_{k+1,6k-2,2}/{\mathbb S}_{6k-2} \to D_2$ over the point $[C]$
consists of one point.
\end{proposition}
\begin{proof}
The decription of the admissible cover representing the unique point
is completely similar to the $D_3$ case. We analyze again the uniqueness.
As before a point in the fibre corresponds to a cover of a rational
curve with two irreducible components ${\PP}_1$ and ${\PP}_2$. Over ${\PP}_2$
we have a cover $C^{\prime} \to {\PP}_2$ and over ${\PP}_1$ a cover $R\to {\PP}_1$ and a number of smooth rational curves $R_j$ mapping with finite degree
to ${\PP}_1$. 

Let $p_1$ be the point over which the ramification of type $(2,2)$ occurs.
If $p_1 \in {\PP}_2$ we find as above that the admissible cover is the one
we want. If $p_1 \in {\PP}_1$ then let $r_1,r_2$ be the ramification points
of type $(2,2)$ over $p_1$. We have the following cases:
\begin{enumerate}
\item{} $r_1,r_2 \in R$;
\item{} $r_1 \in R$, $r_2 \in R_j$ for some $j$;
\item{} $r_1 \in R_{j_1}$ and $r_2 \in R_{j_2}$  for some $j_1\neq j_2$;
\item{} $r_1,r_2 \in R_j$ for some $j$.
\end{enumerate}
In case 1) we conclude that $R \to {\PP}_1$ has degree $\geq 4$, hence the sum
of the ramification degrees at $p$ and $q$ is at least $4$, contradicting the generality of $p$ and $q$.
In case 2) the degree $R_j \to {\PP}_1$ is at least $2$, hence the ramification
degree at $q_j$ with $q_j=R_j \cap C^{\prime}$ is at least $2$, contradicting
lemma \ref{dlemma}. The other cases are easy because in case 3) the ramification degrees of $q_{j_1}$ and $q_{j_2}$ are at least $2$, contradicting the
choice of $(p,q)$, while in case 4) the ramification degree at $q_j$ is
at least $4$ which contradicts the generality of $C^{\prime}$.
Thus we are done in all cases.  
\end{proof}

In order to prove the Theorem we have to analyze the multiplicity.

Our local analysis of the map $\pi :  \overline{H}_{k+1, 6k-2,3} 
\to D_3\subset \overline{M}_g$ over the point $[C]$ is similar to the
one described in \cite{HMu}, pages 76-78 for the case of admissible covers 
with simple branching only. 
For a similar description over Hurwitz schemes of other types, 
see \cite{C}, Section 3 and \cite{D1} p.\ 46.

We take a point $x$ in the fiber of the covering 
$\pi :  \overline{H}_{k+1, 6k-2,3} \to D_3$ over $[C]$. 
As we have seen,  $x$ corresponds to a covering of the form 
$X$ defined above - modulo renumbering of the marked simple branch points - 
and it is a smooth point of the space $\overline{H}_{k+1, 6k-2,3}$.
By the uniqueness we proved above, in a neighborhood of the point $[C]$ 
the variety $D_3$  is the image via the map $\pi$ of a 
neighborhood of the point~$x$. 
We choose a marking of all branch points by 
marking points $p_2,p_3$ on $\PP_1$.
If $\sigma $ is the permutation of $ {\mathbb S}_{2k-2}$ 
interchanging $p_2$ and $p_3$, then $\sigma x=x$, cf.\ Remark 
\ref{hurwitznumber}.
The fixed locus of the permutation  $\sigma $ in the neighborhood of $x$ 
is a divisor $\Delta$.
The complement of $\Delta $ in the neighborhood of $x$ corresponds 
to coverings of smooth curves. 
Therefore, locally at $[C]$, the image of $\Delta $ corresponds to 
the intersection of $D_3$ with the boundary divisor $\Delta _0$  of 
$\overline{M}_g$.   
The map $ \tau: \overline{H}_{k+1, 6k-2,3} \to 
\overline{H}_{k+1, 6k-2,3}/\langle \sigma \rangle$ 
is locally around $x'=\tau(x)$ a degree $2$ covering with ramification 
locus  $\Delta$, see  \cite{HMu}, bottom of p.\ 76.  

As is shown in  \cite{HMu}, p.\ 77, the induced map 
$\lambda : \overline{H}_{k+1, 6k-2,3}/\langle \sigma \rangle \to D_3 
\subset \overline{M}_g$, 
has the property that $\lambda ^*(\Delta_0) = \tau(\Delta) $ with multiplicity one. 
This implies that $D_3$ and $\Delta _0$ meet transversally in the neighborhood 
of $[C]$. Since $[C]$ is locally a generic point of the 
intersection of $D_3$ with $\Delta _0$,  we conclude that it is a smooth 
point of $D_3$. Moreover, since $\lambda^*(\Delta_0) =\tau(\Delta )$ 
with multiplicity one, we find that 
the ramification index of $x'$, which is a generic point,
equals $1$. Hence the ramification index at the point $x$ of the map 
$\pi :  \overline{H}_{k+1, 6k-2,3} \to D_3\subset \overline{M}_g$ is $2$ 
and this finishes the proof of the Theorem for the case of $E_3$.
The analysis for the $E_2$ case is similar. 
\end{section}
\begin{section}{The Calculation of the Class}

We shall now carry out the calculation of the class of $D_2$.
We use the calculation of the class of $D_3$ due to Harris 
in \cite{H}, p.\ 466 and
the formula of Kokotov, Korotkin and Zograf in \cite{KKZ}.
Harris gives the class of $D_3$ (for $k\geq 2$) as
$$
[D_3]= 12 \frac{(2k-3)!}{(k+1)!(k-2)!} \left[
(12 k^2+46k -8) \lambda-b_0 \delta_0 -\sum_{j=1}^k b_j \delta_j \right],
$$ 
with
$b_0=2k^2+4k-1$ and for $b_j=2j (2k-j)(3k+2)$ for $j>0$.
We can rewrite this as
$$
[D_3]= \frac{3}{2k-1} N 
\left[
2(k+4)(6k-1) \lambda-b_0 \delta_0 -\sum_{j=1}^k b_j \delta_j \right],
$$
where $N={2k\choose k+1}/k= {2k \choose k}/(k+1)$.

In their paper \cite{KKZ} Kokotov, Korotkin and Zograf give a formula
for the (first Chern) class $\lambda_H$ of the Hodge bundle on $\bH_{g,d}$
(which is the pull back of the class $\lambda$ of the Hodge bundle
on $\overline{M}_g$).
In our case their formula (Thm.\ 3, formula (3.13)) reads
$$
\lambda_H= \sum_{b_2}^{3k} \sum_{\mu} m(\mu) \left[
\frac{b_2(6k-b_2)}{8(6k-1)}-\frac{1}{12}(k+1-
\sum_i \frac{1}{m_i}) \right] \, \delta^{(b_2)}_\mu,
$$
where $b_2$ is the number of marked point on $\PP_2$, 
$\mu=(m_i)$'s are the ramifications over $s$, $\delta^{(b_2)}_\mu$ 
the corresponding boundary divisor and $m(\mu)$ is the least
common multiple of the $m_i$'s; cf.\ the proof of Thm.\ 3 of loc.\ cit.

We apply the push forward $\pi_*$ to this formula and plug in Harris result.
For $E_0$ we have $k+1$ points over $s$ of ramification degree
$m_i=1$, hence $m(\mu)=1$. For $E_2$ we have $k-3$  points of ramification  
degree $1$ and two of ramification  degree $2$, so $m(\mu)=2$. Similarly, for $E_3$
we have $k-2$ points of ramification  degree $1$ and one of ramification degree $3$, so
$m(\mu)=3$.  
For $E_{j,c}$ we have $k-j+2c$ points over $s$ with 
ramifications degree $1$ and one of ramification degree $j+1-2c$,
so $m(\mu)=j+1-2c$.
This yields:
\begin{proposition}\label{pushforward} $\pi_*(\lambda_H)$
of the Hodge class $\lambda_H$ is given by
$$
\begin{aligned}
&
\frac{(3k-1)}{2(6k-1)}\pi_*[E_0] -\frac{1}{2(6k-1)}\pi_*[E_2]+ 
\frac{3k-5}{6(6k-1)} \pi_*[E_3] +\\
&\sum_{j=1}^{k}\sum_{c=0}^{[j/2]} (j+1-2c)\left[
\frac{(6k-3j)(3j)}{8(6k-1)}-\frac{1}{12}(j+1-2c -\frac{1}{j+1-2c})\right]
\pi_*[E_j,c]\\
\end{aligned}
$$
\end{proposition}

Here we have to interpret the classes $\pi_*[E_0],\ldots, \pi_*[E_{j,c}]$
in the right way since we are working on the stack $\overline{M}_g$.
By applying $\pi_*$ with its degree
$\deg(\pi)=(6k)! N$ and using Proposition \ref{degree} and 
Theorem \ref{degreeE} we find
$$
\pi_*(\lambda_H)=\deg(\pi) \lambda_M, \qquad \pi_*[E_0]=
\frac{\deg(\pi)}{2} \delta_0.
$$
Indeed, a generic admissible cover of $E_0$ admits no non-trivial
automorphisms fixing the marked points. (That the degree of $\pi$
restricted to $E_0$ is $\deg (\pi)/2$ is due to the fact that such
an admissible cover allows an involution that does not fix the
marked points.) Similarly, we find 
$\pi_*[E_3]=\frac{(6k)!}{2} [D_3]$, with the class of 
$D_3$ given above. 
Along $E_2$ an admissible cover has a
${\ZZ}/2{\ZZ} \times {\ZZ}/2{\ZZ}$ in its automorphism group preserving
the marked points with the two generators corresponding to the 
covering involutions on $S_1$ and $S_2$ over ${\PP}_1$ 
(see Prop.\ \ref{description}, i). But locally along $E_2$
the infinitesimal defomation space has a normal crossing singularity,
cf.\ Lemma \ref{deformations}.
We go to the normalization and interpret 
the formula of \cite{KKZ} there (cf.\ 
the remarks at the end of \S (3.1) of \cite{KKZ}). Over $E_2$
this is a $2:1$ cover. 
So taking into account these factors $2/2^2$ of $2$ 
we find $\pi_*[E_2]=(6k)!\, [D_2]/2$.

From Proposition \ref{degree} we get for $j>0$
$$
\pi_*[E_{j,c}]=
\frac{(6k)! \, (j+1-2c)^2}{(j+1)(2k-j+1)}{j+1 \choose c}
{2k-j+1 \choose k+1-c} \, \delta_j.
$$
We put for $i\in {\ZZ}_{\geq 1}$
$$
A_i(j)=\frac{1}{(j+1)(2k+1-j)}\sum_{c=0}^{[j/2]} (j+1-2c)^i{j+1 \choose c}
{2k-j+1 \choose k+1-c}\, .
$$
Then with $N={2k \choose k+1}/k$ we have
$$
A_2(j)=N, \qquad 
A_4(j)= \left(1+\frac{3j(2k-j)}{2k-1}\right) N,
$$
while for $A_3(j)$ we get if $j$ is even
$$
A_3(j)=\frac{j(2k-j)+k}{k(k+1)} {j \choose [j/2]} {2k-j\choose k-[j/2]}
$$
and for $j$ odd 
$$
A_3(j)=\frac{(j+1)(2k-j)}{k(k+1)} {j+1 \choose 1+[j/2]}{2k-j-1 \choose k-1-[j/2]}\, .
$$

By multiplying by $2(6k-1)$ and bringing $\pi_*[E_2]$ to the other side
in the equation for $\pi_*(\lambda_H)$ in Proposition \ref{pushforward} 
we get
$$
\begin{aligned}
\pi_*[E_2]= & 
-2(6k-1)\pi_*(\lambda_H)+\frac{3k-5}{3}\pi_*[E_3] +(3k-1)\pi_*[E_0]+ \\
&(6k)! \sum_{j=1}^k  \left[ 
\frac{(6k-3j)(3j)}{4} A_3(j) +\frac{6k-1}{6} (-A_4(j) +A_2(j)) \right]
\delta_j \, .\\
\end{aligned}
$$
Dividing by $(6k)!$ we find
$$
\begin{aligned}
 {[ D_2]}/2= 
&  -2(6k-1)N\lambda_M + \frac{3k-5}{6} [D_3] +\frac{3k-1}{2}N \delta_0 + \\
 &  \sum_{j=1}^k  \left[
 \frac{(6k-3j)(3j)}{4} A_3(j) +\frac{6k-1}{6} (-A_4(j) +A_2(j)) \right]
\delta_j\, . \\
\end{aligned}
$$
Only the first two terms on the right hand side contribute to
the coefficient of $\lambda_M$ and the contribution is
$$
\begin{aligned}
-2(6k-1)N \lambda_M+ \frac{3k-5}{6} \frac{3}{2k-1} N 2 (k+4)(6k-1)\lambda_M=&\\
3N \frac{6k-1}{2k-1}  (k-2)(k+3)\lambda_M \, .&
\\
\end{aligned}
$$
The coefficient of $\delta_0$ comes from the second and third term on the right hand
side. It is
$$
-\frac{3k-5}{6}\frac{3}{2k-1} N (2k^2+4k-1) + (3k-1)\frac{N}{2}= 
-\frac{N}{2k-1} (k-2)(3k^2+4k-1).
$$
The coefficient of $\delta_j$, $j\geq 1$, comes from the second and fourth term on the right hand
side. We get
$$
\begin{aligned}
-\frac{1}{2} N \frac{3k-5}{2k-1} 2j (2k-j) (3k+2) +
\frac{9(2k-j)j}{4} A_3(j) - N\frac{(6k-1)j(2k-j)}{2(2k-1)}= \\ 
-\frac{3\, N j (2k-j)}{2 (2k-1)} (6k^2-4k-7) + \frac{9}{4} j (2k-j) A_3(j).
\end{aligned}
$$
This concludes the proof of the theorem.
\end{section}
\begin{section}{A Final Check}
Our main result reproduces the well-known relation 
$10\lambda=\delta_0+\delta_1$ for $k=1$ and gives zero for $k=2$,
as it should.
It also satisfies the relation $c_{\lambda}+12\, c_0 -c_1=0$ given in
Lemma 3.1 of \cite{H}. But these checks using homogeneous linear relations
leave the possibility of common factor in the coefficients 
$c_{\lambda},c_0,\ldots,c_k$.
To rule this out we consider a test curve in $\overline{M}_g$.
Take a general curve $B$ of genus $g-1$, a general point $p\in B$
and identify in the blow-up of $B\times B$ at $(p,p)$ the diagonal
with the section $\{ p \} \times B$. This gives a family $\pi: S \to B$
of one-nodal curves with 
$$
B\cdot \lambda=0, \quad B\cdot \delta_0= 2-2g, \quad
B \cdot \delta_1=1,\quad  \text{and $B\cdot \delta_j=0$ for $j\geq 2$}.
$$
\begin{lemma}\label{testlemma}
We have $B \cdot D_2 = (k-1)(k-2)(12\, k + 10)N$.
\end{lemma}
\begin{proof}
We have $B\cdot D_2= S_p + 2\, S_p^{\prime}$ with  $S_p$ and $S_p^{\prime}$ 
defined in the proof of Lemma \ref{dlemma}. The argument is similar to that of
\cite{H}, Lemma 3.9. Set-theoretically we have $D_2 \cdot B= S_p\cup S_p^{\prime}$.  The argument for the multiplicity of $S_p$ is similar to that of
loc.\ cit. As to the multiplicity of $S_p^{\prime}$, 
an analysis shows that it equals $2$ due to the involution 
involved here.
In the proof of Lemma \ref{dlemma}
we gave the cardinalities: $\#S_p= k\, d(k)$ and 
$\# S_p^{\prime}=(k-2)\, c(k)$. This proves the Lemma.
\end{proof}
On the other hand using the intersection numbers of $B$ with
$\lambda$ and the $\delta_i$ we get by our Theorem
$$
B \cdot D_2 = -2(2k-1)c_0+c_1 = 2(k-1)(k-2) (6k+5) N
$$
in perfect agreement with the Lemma \ref{testlemma}.
\end{section}


 \end{document}